\documentclass[12]{amsart}
\usepackage{amscd,amssymb,mathabx}
\usepackage[all]{xy}
\usepackage{graphicx,calrsfs}
\usepackage[colorlinks,plainpages,backref,urlcolor=blue]{hyperref}

\newtheorem{theorem}{Theorem}[section]
\newtheorem{corollary}[theorem]{Corollary}
\newtheorem{lemma}[theorem]{Lemma}
\newtheorem{prop}[theorem]{Proposition}

\theoremstyle{definition}
\newtheorem{definition}[theorem]{Definition}
\newtheorem{example}[theorem]{Example}
\newtheorem{remark}[theorem]{Remark}
\newtheorem*{ack}{Acknowledgment}

\newcommand{\N}{\mathbb{N}}
\newcommand{\Z}{\mathbb{Z}}

\newcommand{\C}{\mathbb{C}}

\newcommand{\TT}{\mathbb{T}}
\newcommand{\PP}{\mathbb{P}}
\newcommand{\GG}{\mathbb{G}}

\DeclareMathAlphabet{\pazocal}{OMS}{zplm}{m}{n}

\newcommand{\XX}{{\pazocal X}}

\newcommand{\CC}{{\pazocal{C}}}

\newcommand{\RR}{{\mathcal R}}
\newcommand{\VV}{{\mathcal V}}

\newcommand{\F}{{\mathcal{F}}}

\newcommand{\g}{{\mathfrak{g}}}

\newcommand{\h}{{\mathfrak{h}}}
\newcommand{\kk}{{\mathfrak{k}}}
\newcommand{\gl}{{\mathfrak{gl}}}
\renewcommand{\sl}{{\mathfrak{sl}}}

\renewcommand{\ss}{{\mathfrak{s}}}

\newcommand{\n}{{\mathfrak{n}}}

\newcommand{\G}{{\Gamma}}

\DeclareMathOperator{\rank}{rank}

\DeclareMathOperator{\im}{im}

\DeclareMathOperator{\id}{id}
\DeclareMathOperator{\ab}{{ab}}

\DeclareMathOperator{\GL}{GL}
\DeclareMathOperator{\SL}{SL}

\DeclareMathOperator{\PSL}{PSL}
\DeclareMathOperator{\Hom}{{Hom}}

\DeclareMathOperator{\Der}{{Der}}
\DeclareMathOperator{\ad}{ad}

\DeclareMathOperator{\cone}{cone}
\DeclareMathOperator{\rep}{Rep}
\DeclareMathOperator{\eig}{Eigen}

\newcommand{\same}{\Longleftrightarrow}
\newcommand{\surj}{\twoheadrightarrow}
\newcommand{\inj}{\hookrightarrow}

\newcommand{\abs}[1]{\left| #1 \right|}

\def\dot{\mathchar"013A}
\newcommand{\hdot}{{\raise1pt\hbox to0.35em{\Huge $\dot$}}}
\newcommand{\cdga}{\ensuremath{\mathsf{cdga}}}

\begin{document}

\title[Rank two jump loci for solvmanifolds and Lie algebras]{%
Rank two jump loci for solvmanifolds and Lie algebras}

\author[S.~Papadima]{\c Stefan Papadima$^1$}
\address{Simion Stoilow Institute of Mathematics,
P.O. Box 1-764, RO-014700 Bucharest, Romania}
\email{Stefan.Papadima@imar.ro}
\thanks{$^1$Partially supported by PN-II-ID-PCE-2011-3-0288, grant 
132/05.10.2011}

\author[L.~Paunescu]{Lauren\c tiu Paunescu}
\address{School of Mathematics and Statistics, The University of Sydney,  Sydney, 2006, Australia}
\email{laurent@maths.usyd.edu.au}

\subjclass[2010]{Primary
55N25.
Secondary
17B56, 20J06, 22E40.}

\keywords{Representation variety, characteristic variety, resonance variety, 
analytic germ, solvmanifold, Lie algebra.}

\begin{abstract}
We consider representation varieties in $\SL_2$ for lattices in solvable Lie groups,
and representation varieties in $\sl_2$ for finite-dimensional Lie algebras.
Inside them, we examine depth 1 characteristic varieties for solvmanifolds, 
respectively resonance varieties for cochain Differential Graded Algebras of Lie algebras.
We prove a general result that leads, in both cases, to the complete description of 
the analytic germs at the origin, for the corresponding embedded rank 2 jump loci.
\end{abstract}

\maketitle
\setcounter{tocdepth}{1}

\section{Introduction and statement of results}
\label{sec:intro}

Jump loci are basic objects in geometry, topology and group theory. They appear in the following way.
Let $M$ be a connected, compact differentiable manifold (up to homotopy), with fundamental group $\pi$. Let
$\iota : \GG \to \GL (V)$ be an algebraic representation of complex linear algebraic groups. 
The representation variety $\Hom (\pi, \GG)$ is an affine variety with distinguished basepoint $1$,
the trivial group homomorphism. It encodes the representation theory of $\pi$ into $\GG$. Furthermore,
it comes equipped with a collection of closed subvarieties: the {\em characteristic varieties} (or jump loci)
in degree $i$ and depth $r$
\begin{equation}
\label{eq:defv}
\VV^i_r(M, \iota)=\{\rho \in \Hom(\pi, \GG) \mid
\dim H^i(M, {}_{\iota\rho}V)\ge r\} \, ,
\end{equation}
which stratify the local systems on $M$ with typical fiber $V$ that factor through $\iota$, according to
the dimension of the corresponding twisted cohomology of $M$. This rich structure contains valuable
information on the geometry and topology of $M$. 

The {\em rank one} case (when $\iota=\id_{\C^{\times}}$) received a lot of attention. The
corresponding characteristic varieties, denoted simply by $\VV^i_r (M)$, are the topological
counterpart, for $r=1$, of the loci defined in algebraic geometry by Green and Lazarsfeld in \cite{GL1, GL2},
when $M$ is a complex projective manifold. Results of Arapura \cite{A}, further refined in  \cite{DPS-duke}, 
show how the regular maps from $M$ onto curves of general type may be recovered from the 
geometry of the variety $\VV^1_1 (M)$, when $M$ is supposed merely quasi-projective. In topology,
it is known that the rank $1$ jump loci  $\VV^i_1 (M)$ control delicate finiteness properties of
Alexander-type invariants of $M$. This fact was used in \cite{DP-ann, DHP} to obtain significant results
on certain important subgroups of mapping class groups of closed Riemann surfaces.

The {\em rank two} case (when $\GG$ is $\SL_2 (\C)$ or $\PSL_2 (\C)$) is considerably more difficult.
Indeed, the universality theorem of Kapovich and Millson \cite{KM} shows that the varieties $\Hom (\pi, \PSL_2 (\C))$
have arbitrarily bad singularities away from $1$, when $\pi$ runs through the family of Artin groups. 
At the same time, the large number of applications of this case makes it very interesting. For example, a 
recent result from \cite{PS14} establishes a connexion between the rank $2$ case and a difficult open
problem about the monodromy action on the Milnor fiber homology of a complex hyperplane arrangement.

In this note, we approach analytic germs at $1$ of representation varieties and jump loci through the prism of
the main result from \cite{DP-ccm}. The idea is to replace $M$ by a {\em model} $A=(A^\hdot, d)$:
a connected, finite-dimensional Commutative Differential Graded Algebra (for short, a $\cdga$) having 
the same Sullivan minimal model \cite{S} as the complex de Rham $\cdga$ of $M$. Denote by 
$\theta\colon \g\to \gl(V)$ the tangent map at $1$ of $\iota$. Replace the $\GG$-representation variety of $\pi$ by
the affine variety of $\g$-valued {\em flat connections} on $A$, $\F(A,\g)\subseteq A^1\otimes \g$, with
natural basepoint $0$. For $\omega \in \F(A,\g)$, consider the cochain complex $(A^\hdot \otimes V, d_{\omega})$,
where $d_{\omega}$ is the associated algebraic covariant derivative. Replace $\VV^i_r(M, \iota)$ in \eqref{eq:defv}   by the 
{\em resonance varieties} $\RR^i_r(A, \theta)$ defined in \eqref{eq:defres}. These Zariski closed subvarieties of
$\F(A,\g)$ record how $\dim H^i(A \otimes V, d_{\omega})$ jumps, for $\omega \in \F(A,\g)$. See Section \ref{sec:gral}
for more details. 

In what follows, we will denote by $\XX_{(x)}$ the analytic germ at $x$ of an affine variety $\XX$. 
Theorem B from \cite{DP-ccm} establishes an isomorphism $\Hom (\pi, \GG)_{(1)} \simeq \F(A,\g)_{(0)}$
that identifies $\VV^i_r(M, \iota)_{(1)}$ with $\RR^i_r(A, \theta)_{(0)}$ for all $i,r$, when $A$ models $M$. We
are going to examine {\em topological and algebraic Green--Lazarsfeld loci}, corresponding to $r=1$. In
the rank $1$ case, we will drop $\iota$ and $\theta$ from notation.

To state our first main result, we need to recall two constructions from \cite{MPPS} (see also Section \ref{sec:gral}).
The Zariski closed subset $\F^1 (A,\g)\subseteq \F(A,\g)$ consists of those flat connections of the form 
$\omega= \eta\otimes g$ with $d\eta =0$. The additional condition $\det \theta (g)=0$ defines the 
Zariski closed subset $\Pi(A,\theta)\subseteq \F^1(A,\g)$.

\begin{theorem}[Corollary \ref{cor:pisharp}]
\label{thm:intro1}
Let $A=(A^\hdot, d)$ be a connected, finite-dimensional $\cdga$ over $\C$. Let $\theta\colon \g\to \gl(V)$
be a finite-dimensional representation of a finite-dimensional Lie algebra $\g$, with $V\ne 0$. Assume that $0$
is an isolated point of $\bigcup_{i\ge 0} \RR^i_1(A)$, and $\F(A,\g)_{(0)}= \F^1 (A,\g)_{(0)}$. Then
$\RR^i_1(A, \theta)_{(0)}= \Pi(A,\theta)_{(0)}$, for all $i$. 
\end{theorem}

In Section \ref{sec:solv}, we apply this general result to a discrete, co-compact subgroup $\G$ of a $1$-connected
solvable real Lie group $S$ and the associated aspherical solvmanifold $M=S/\G$. It leads to the following complete 
description of embedded germs at $1$  of rank two topological Green--Lazarsfeld loci, for solvmanifolds. (In what follows,
we denote by $\PP (H)$ the projective space of a vector space, and by $V(f)$ the zero set of a polynomial.)

\begin{theorem}[Theorem \ref{thm:jumpsol}]
\label{thm:intro2}
Let $M=S/\G$ be a solvmanifold and let $\iota : \GG \to \GL (V)$ be a rational representation of a complex,
semisimple linear algebraic group of rank $1$, with tangent map $\theta\colon \sl_2 \to \gl(V)$. Then 
the germ at $1$ of $\Hom (\G, \GG)$ is isomorphic to the germ at $0$ of the cone on $\PP(H^1(M))\times \PP(\sl_2)$,
and the embedded topological Green--Lazarsfeld germs are given by
\begin{equation*}
\VV^i_1(M, \iota)_{(1)} =
\begin{cases}
\emptyset & \text{if $H^i (M)=0$,}\\
\cone (\PP(H^1(M))\times V(\det \circ \theta))_{(0)}
& \text{otherwise.}
\end{cases}
\end{equation*}
\end{theorem}

This extends the computation done in \cite{MPPS} for nilmanifolds, by replacing Nomizu's model \cite{N} of
such a manifold with a model found by Kasuya \cite{K1} for an arbitrary solvmanifold. As explained in \cite{R},
lattices in solvable Lie groups are much more complicated than those in nilpotent Lie groups. Nevertheless,
Theorem \ref{thm:intro1} makes things work well. Note that the first assumption from Theorem \ref{thm:intro1} 
(on rank one resonance) also appears in Theorem $1$ from \cite{DPS14}: it is equivalent to the fact that all
completed Alexander invariants of $M$ are finite-dimensional, for an arbitrary compact manifold $M$ modeled by $A$.
To check it for solvmanifolds, we prove in Theorem \ref{thm:charpoly} a result of independent interest: 
each rank one characteristic variety $\VV^i_1(M)$ is finite, when $M$ is the classifying space of a virtually polycyclic group. 

Next, we recall that Nomizu's model for the nilmanifold $M=S/\G$ is the Chevalley--Eilenberg cochain $\cdga$,
$\CC^{\hdot}(\ss \otimes \C)$, of the nilpotent Lie algebra $\ss \otimes \C$, where $\ss$ is the Lie algebra of $S$. 
In Section \ref{sec:lie}, we apply
Theorem \ref{thm:intro1} to another large class of examples: cochain $\cdga$'s of arbitrary 
finite-dimensional Lie algebras. Denoting by $H^\hdot (\h)$ untwisted Lie algebra cohomology, we obtain in
Theorem \ref{thm:jumplie} the following complete description of embedded germs at $0$ of rank two algebraic
Green--Lazarsfeld loci, for this class. 

\begin{theorem}[Theorem \ref{thm:jumplie}]
\label{thm:intro3}
Let $\h$ be a finite-dimensional complex Lie algebra and let $\theta\colon \sl_2 \to \gl(V)$ be
a finite-dimensional Lie representation with $V\ne 0$. Then $\F (\CC^\hdot \h, \sl_2)_{(0)}$ is
equal to $\{ 0\}$ when $H^1 (\h)=0$, and otherwise is isomorphic to $\cone (\PP(H^1(\h))\times \PP(\sl_2))_{(0)}$.
In the first case, $\RR^i_1( \CC^\hdot \h, \theta)_{(0)}$ is empty or equal to  $\{ 0\}$, depending on
whether $H^i (\h)$ vanishes or not. In the second case, the embedded germs of depth $1$ resonance varieties are given by
\begin{equation*}
\RR^i_1(\CC^\hdot \h, \theta)_{(0)} =
\begin{cases}
\emptyset & \text{if $H^i (\h)=0$,}\\
\cone (\PP(H^1(\h))\times V(\det \circ \theta))_{(0)}
& \text{otherwise.}
\end{cases}
\end{equation*}
\end{theorem}

When $A^\hdot =\CC^\hdot \h$, the analysis of the second hypothesis from Theorem \ref{thm:intro1}
(on flat connections) turns out to be the more difficult part. We know from \cite {MPPS} that 
$\F (A, \sl_2)= \F^1 (A, \sl_2)$, when $\h$ is nilpotent. We first extend this property to the solvable case,
where we prove that $\F (\CC^\hdot \h, \sl_2)_{(0)}=\F^1 (\CC^\hdot \h, \sl_2)_{(0)}$. Along the way, 
we show in Lemma \ref{lem:metab}\eqref{m2} that the global property 
$\F (\CC^\hdot \h, \sl_2)=\F^1 (\CC^\hdot \h, \sl_2)$ actually characterizes the nilpotence of $\h$,
within a certain metabelian class. Finally, we associate to an arbitrary finite-dimensional Lie algebra $\h$
a solvable Lie algebra $\widetilde{\ss}$, and we prove in Proposition \ref{prop:homgerms} that,
for any finite-dimensional Lie algebra $\kk$, there is a natural analytic isomorphism,
$\F (\CC^\hdot \h, \kk)_{(0)} \simeq \F (\CC^\hdot \, \widetilde{\ss}, \kk)_{(0)}$. This leads to the equality
$\F (\CC^\hdot \h, \sl_2)_{(0)}=\F^1 (\CC^\hdot \h, \sl_2)_{(0)}$, in the general case.

\section{Vanishing rank one resonance}
\label{sec:gral}

We start by isolating a useful property of rank $1$ resonance. In the sequel, this will enable us 
to treat simultaneously the germs at the origin of rank $2$ jump loci, for both solvmanifolds and Lie algebras.

First, we need to review a couple of basic definitions and facts related to algebraic jump loci, 
following \cite{DP-ccm} and \cite{MPPS}. Unless otherwise mentioned, our ground ring will be $\C$.

Let $A=(A^\hdot, d)$ be a $\cdga$, whose defining axioms capture the essential algebraic features 
of the de Rham algebra of a differentiable manifold. Our standard assumption is that $A$ is
finite-dimensional (as a vector space) and connected (i.e., $A^0=\C\cdot 1$). For a finite-dimensional Lie 
algebra $\g$, we denote by $\F(A,\g)\subseteq A^1\otimes \g$ the variety of $\g$-valued flat connections on $A$,
given by the solutions of the Maurer--Cartan equation, $d\omega + \frac{1}{2} [\omega, \omega]=0$. 
This construction gives a Zariski closed subset  of $A^1\otimes \g$ containing the distinguished point $0$,
and is natural in both $A$ and $\g$. When $\g$ is abelian, $\F(A,\g)= H^1(A)\otimes \g$, since $A$ is connected. Otherwise,
the singularity $\F(A,\g)_{(0)}$ can be pretty complicated. We will denote $\F(A,\C)$ simply by $\F(A)$.

Now, let $\theta\colon \g \to \gl(V)$ be a finite-dimensional Lie representation with $V\ne 0$. 
For $\omega \in \F(A,\g)$, the covariant derivative $d_{\omega}$ acts on $A^\hdot \otimes V$ (via $\theta$),
in the following way. Write $\omega =\sum_i \eta_i \otimes g_i \in A^1\otimes \g$. Then 
$d_{\omega}(\eta \otimes v)=d\eta \otimes v+\sum_i \eta_i \eta\otimes \theta(g_i)v$, for $\eta \otimes v\in A^\hdot \otimes V$.
The flatness condition implies that $d_{\omega}^2=0$. We may thus speak for each degree $i$ 
about the descending filtration of $\F(A,\g)$ by the depth $r$ {\em resonance varieties}
\begin{equation}
\label{eq:defres}
\RR^i_r(A, \theta)= \{\omega \in \F(A,\g)
\mid \dim H^i(A \otimes V, d_{\omega}) \ge  r\}\, ,
\end{equation}
which are Zariski closed in $\F(A,\g)$. We are going to pay particular attention to the depth $1$ resonance 
varieties, since their germs at $0$ are isomorphic to the germs at $1$ of the corresponding 
topological Green--Lazarsfeld loci $\VV^i_1(M, \iota)$, when $A$ models $M$. 

The simplest case is the {\em rank one} case, corresponding to $\theta=\id_{\C}$, when 
$\RR^i_r(A, \theta)\subseteq  \F(A,\C)$ is denoted simply by $\RR^i_r(A)\subseteq  \F(A)$,
and $\F(A)= H^1(A)$. This case is much more studied 
than the non-abelian situation. Our starting point in this note is to identify properties of rank $1$
resonance that give valuable information on higher rank resonance. A simple useful remark in this direction, 
made in \cite{MPPS}, is that
\begin{equation}
\label{eq:germs12}
\RR^i_1(A, \theta)_{(0)} \neq \emptyset \same \RR^i_1(A)_{(0)} \neq \emptyset \same H^i(A)\neq 0\, .
\end{equation}

More information comes from considering the quadratic map
\begin{equation}
\label{eq:pmap}
P\colon A^1 \times \g \longrightarrow A^1 \otimes \g \, , \quad (\eta, g) \mapsto \eta\otimes g\, ,
\end{equation}
the induced map $P\colon H^1(A) \times \g \longrightarrow \F(A, \g)$, and the {\em (essentially) rank one}
part $\F^1(A, \g):= P(H^1(A) \times \g)$, which is Zariski closed in $\F(A, \g)$ and contains the origin $0$.
Define $\Pi(A,\theta):= P(H^1(A) \times V(\det \circ \theta))$, where $\det :\gl (V) \to \C$ is the deteminant.
Again, $\Pi(A,\theta)$ is Zariski closed in $\F^1(A, \g)$ and contains  $0$. Moreover, as shown in \cite{MPPS},
$H^i(A)\ne 0$ implies that
\begin{equation}
\label{eq:pibound}
\Pi(A,\theta) \subseteq \RR^i_1(A, \theta)\, .
\end{equation}

This higher rank resonance bound actually follows from a more precise result.

\begin{theorem}[\cite{MPPS}]
\label{thm:main0}
Let $\omega=\eta \otimes g$ with $d\eta=0$ be an arbitrary element of $\F^1(A, \g)$.
Then $\omega\in \RR^i_1(A, \theta)$ if and only if there is an eigenvalue $\lambda$ of $\theta (g)$
such that $\lambda \eta\in \RR^i_1(A)$. 
\end{theorem}

In coordinates, consider the quadratic map
\begin{equation}
\label{eq:pcoord}
P\colon \C^m\times \C^n\longrightarrow \C^{mn}\equiv \C^m\otimes \C^n\, , \quad (x,y)\mapsto z\, ,
\end{equation}
where $z_{ij}= x_iy_j$. We denote its image by $\F^1(m,n)$. Clearly, $\F^1(m,n)= \F^1(A,\g)$, 
when $\dim H^1(A)=m$ and $\dim \g=n$. 
As is well-known, $\F^1(m,n)\simeq \cone(\PP^{m-1} \times \PP^{n-1})$
is Zariski closed in $\C^m\otimes \C^n$.

Clearly, $P^{-1}(0)=  \C^m\times 0 \bigcup 0\times \C^n$.
To describe the other fibers, let us consider the $\C^{\times}$-action defined by
$t\cdot (x,y)= (t^{-1}x, ty)$, for $t\in \C^{\times}$. Then $P(x,y)=P(x',y')$ if and only if
$(x,y)$ and $(x',y')$ lie in the same $\C^{\times}$-orbit, when $P(x,y)\ne 0$. 

We will use the following nice property of the regular map \eqref{eq:pcoord}.

\begin{lemma}
\label{lem:pnice}
A regular function $f$ on $\C^m\times \C^n$ factors through the surjection 
$P\colon \C^m\times \C^n \surj \F^1(m,n)$ if and only if $f$ is $\C^{\times}$-invariant.
If this is the case, then the (unique) quotient function $F$ is regular on $\F^1(m,n)$.
\end{lemma}

\begin{proof}
If $f$ factors, then clearly $f$ is $\C^{\times}$-invariant, since $P(t\cdot (x,y))=P(x,y)$,
for all $t\in \C^{\times}$ and $(x,y)\in \C^m\times \C^n$. Conversely, $\C^{\times}$-invariance
means that the polynomial $f$ is of the form $f(x,y)= \sum_{\abs{I}=\abs{J}} c_{IJ}x_Iy_J$
where $(I,J)\in \N^m\times \N^n$. This plainly implies factorization, according to the above
description of the fibers of $P$. The second claim follows from the fact that the algebra of
$\C^{\times}$-invariant polynomials is generated by the monomials $x_iy_j$. 
\end{proof}

\begin{definition}
\label{def:zerores}
We say that the $\cdga$ $A$ has {\em trivial resonance} in degree $i$ if $0$ is an isolated point of
$\RR^i_1(A)$. In other words, $\RR^i_1(A)= \{ 0\} \cup \bigcap_{\alpha} (V(\varphi_{\alpha}) \cap H^1(A))$,
where $\{ \varphi_{\alpha}\}$ are polynomial functions on $A^1$ and there is $\varphi_0$
such that $\varphi_0 (0)\ne 0$. 
\end{definition}

This property has the following significant higher rank consequence.

\begin{theorem}
\label{thm:trivres}
If $A$ has trivial resonance in degree $i$, then $\Pi(A,\theta)$ and $\RR^i_1(A, \theta)\cap \F^1 (A,\g)$ 
have the same germ at $0$, for any representation $\theta\colon \g\to \gl(V)$.
\end{theorem}

\begin{proof}
By \eqref{eq:germs12} and \eqref{eq:pibound}, we know that 
$0\in \Pi(A,\theta) \subseteq \RR^i_1(A, \theta)\cap \F^1 (A,\g)$. Set $p=\dim V$ and consider the polynomial 
function on $A^1\times \C^p$, $\widetilde{f}(\eta, \lambda)= \prod_{i=1}^p \varphi_0 (\lambda_i \eta)$,
where $\varphi_0$ is chosen as in definition \ref{def:zerores}. Since $\widetilde{f}$ is symmetric in
$\lambda$, $\widetilde{f}(\eta, \lambda)=\overline{f}(\eta, \sigma)$, for some polynomial function $\overline{f}$ 
on $A^1\times \C^p$,  when $\sigma_i$ is the $i$-th elementary symmetric function $\sigma_i(\lambda)$, for
$i=1,\dots,p$. Next, write $\det (t\cdot \id -\theta(g))=t^p+ \sum_{i=1}^p (-1)^i c_i t^{p-i}$, where each
$c_i$ is a polynomial function in $g\in \g$.  Clearly, $c_i(g)=\sigma_i(\lambda)$, for all $i$, where 
$\lambda_1, \dots, \lambda_p$ are the eigenvalues of $\theta (g)$. Define the polynomial function $f$ on
$A^1\times \g$ by $f(\eta, g)= \overline{f}(\eta, c(g))$. By construction, $f(\eta, g)= \widetilde{f}(\eta, \lambda)$,
when $\lambda_1, \dots, \lambda_p$ are the eigenvalues of $\theta (g)$. 

We may thus apply Lemma \ref{lem:pnice} to the quadratic map \eqref{eq:pmap} and the regular function $f$. 
Let us check that $f$ factors through $P$. If $\eta=0$ or $g=0$, then clearly $f(\eta, g)=\varphi_0(0)^p$. 
It is equally clear that $f(t^{-1}\eta, tg)= f(\eta, g)$. Hence, we may find a regular function $F$ on $A^1\otimes \g$
such that $F\circ P=f$ and $F(0)=\varphi_0(0)^p \ne 0$. 

Our claim will follow from the inclusion $\RR^i_1(A, \theta)\cap \F^1 (A,\g) \subseteq \Pi(A,\theta)\bigcup V(F)$.
To prove this inclusion, assume that $\omega \in \RR^i_1(A, \theta)$, where $\omega =\eta \otimes g$ with $d\eta =0$,
$\eta\ne 0$ and $\det \theta(g)\ne 0$. We infer from Theorem \ref{thm:main0} that there is an eigenvalue $\lambda_j \ne 0$
of $\theta (g)$ such that $\lambda_j \eta\in \RR^i_1(A)$. Therefore, $\varphi_0 (\lambda_j \eta)=0$, 
according to definition \ref{def:zerores}. By the construction of $f$, $0= \widetilde{f}(\eta, \lambda)=f(\eta, g)$.
Finally, the factorization property of $f$ implies that $f(\eta, g)=F(\omega)$, and we are done.
\end{proof}

Under an additional hypothesis, the above result implies that the bound \eqref{eq:pibound} is sharp,
at the level of germs at the origin.

\begin{corollary}
\label{cor:pisharp}
If $A$ has trivial resonance in degree $i$ and $\F(A,\g)_{(0)}= \F^1 (A,\g)_{(0)}$, then
$\RR^i_1(A, \theta)_{(0)}= \Pi(A,\theta)_{(0)}$, for any representation $\theta\colon \g\to \gl(V)$. 
\end{corollary}

The trivial resonance property is also related to delicate finiteness properties of Alexander invariants of spaces
\cite{PS-mrl, DPS14}. For $\cdga$'s with $d=0$, it follows from definition \eqref{eq:defres} that all resonance
varieties $\RR^i_r(A, \theta)$ are homogeneous. In particular, in this case, triviality of resonance in degree $i$ is
equivalent to $\RR^i_1(A)=\{ 0\}$. 
As shown in \cite{PS-mrl}, examples of this sort abound. More precisely, the $\cdga$'s with $d=0$
and fixed Betti numbers $b_i=\dim A^i$ may be viewed as the points of an affine variety. On this
parameter space, triviality of resonance in degree $i$ is a Zariski open condition, when $b_i>0$. Moreover, for $i=1$
and $b_2 \ge 2b_1 -3\ge -1$, this open set is non-void, as follows from Theorem 1.1 in \cite{PS-crelle}. 
In the next sections, we will exhibit
new classes of examples, with trivial resonance and non-zero differential. 

\begin{example}
\label{ex:free}
For $n\ge 2$, let $M= \C\setminus \{ \text{$n$ points} \}$ be the classifying space of a finitely generated free
non-abelian group. It is well-known that in this case $A=(H^\hdot(M), d=0)$ is a model of $M$, and
$\RR^1_1(A)=\C^n$. Hence, $A$ does not have trivial resonance in degree $1$. 

This simple example also shows that the second assumption from Corollary \ref{cor:pisharp} is not always satisfied.
Indeed, for an arbitrary $\cdga$ $A=(A^\hdot, d)$ and $\omega =\sum_i \eta_i \otimes g_i \in A^1 \otimes \g$,
the Maurer--Cartan equation is equivalent to 
\begin{equation}
\label{eq:mc}
\sum_{i} d\eta_i \otimes g_i +
\sum_{i<j} \eta_i \eta_j \otimes [g_i, g_j] =0.
\end{equation}

In our example, \eqref{eq:mc} shows that $\F(A,\g)= A^1 \otimes \g$. When $\dim \g>1$, 
it follows that $\F(A,\g)_{(0)} \ne \F^1 (A,\g)_{(0)}$.
\end{example}

\section{Solvmanifolds}
\label{sec:solv}

In this section, we apply the general theory to solvmanifolds, where it leads to a complete description
of germs at $1$ for both varieties of $\SL_2$-representations and topological Green--Lazarsfeld loci.
We also show that the rank one topological Green--Lazarsfeld sets of virtually polycyclic groups are finite.

A {\em solv-lattice} is a discrete, co-compact subgroup $\G$ of a $1$-connected solvable real Lie group $S$,
giving rise to the compact, aspherical solvmanifold $M=S/\G$, with fundamental group $\G$. For such a
manifold, Kasuya constructed in \cite{K1} a connected, finite-dimensional $\cdga$ model. Since the
details are rather complicated, we are going to extract only the properties of the {\em Kasuya model} $A$
of $M$ that are relevant to our study, following \cite{K1,K3}. 

Let $\n$ be the nilshadow of the solvable Lie algebra $\ss$ of $S$. The model $A^\hdot$ is a sub-$\cdga$
of the Chevalley-Eilenberg cochain algebra of the finite-dimensional nilpotent Lie algebra $\n\otimes \C$,
$B^\hdot:= \CC^{\hdot}(\n\otimes \C)$. This information is enough to prove that $A$ satisfies the 
second assumption from Corollary \ref{cor:pisharp} globally, for $\g=\sl_2$. 

\begin{lemma}
\label{lem:kflat}
If $A$ is a Kasuya solvmanifold model and $\g=\sl_2$, then $\F(A,\g)=\F^1 (A,\g)$.
\end{lemma}

\begin{proof}
Since $A^\hdot \subseteq B^\hdot$, equation \eqref{eq:mc} readily implies that
$\F(A,\g)=\F(B,\g) \cap A^1\otimes \g$, for any $\g$. Similarly, $\F^1(A,\g)=\F^1(B,\g) \cap A^1\otimes \g$.
Finally, the nilpotence of $\n\otimes \C$ implies, for $\g=\sl_2$, the equality $\F(B,\g)=\F^1 (B,\g)$ \cite{MPPS}.
Our claim follows.
\end{proof}

To verify triviality of resonance for Kasuya models, we will use a general result from \cite{DPS14},
which we now recall. Let $M$ be a connected, finite CW-complex, with fundamental group $\pi$ and
universal abelian cover $M^{\ab}$. The action of $\pi_{\ab}$ by deck-transformations makes $H_{\hdot}(M^{\ab})$ 
a graded $\C [\pi_{\ab}]$-module, with $I$-adic completion denoted $\widehat{H}_{\hdot}(M^{\ab})$,
where $I$ is the augmentation ideal of the group ring $\C [\pi_{\ab}]$. Let $A$ be a $\cdga$ model of 
a connected, compact manifold $M$. Then $0$ is an isolated point of $\bigcup_{i\ge 0} \RR^i_1(A)$
if and only if the $\C$-vector space $\widehat{H}_{\hdot}(M^{\ab})$ is finite-dimensional. When
$M=S/\G$ is a solvmanifold, note that $H_{\hdot}(M^{\ab})= H_{\hdot}(\G')$, where $\G'$  is
the derived subgroup. 

\begin{lemma}
\label{lem:kres}
All connected covers of a solvmanifold $M=S/\G$ have finite dimensional homology. 
Moreover, if $H^i(M)\ne 0$, then $A$ has trivial resonance 
in degree $i$, for an arbitrary $\cdga$ model $A$ of $M$.
\end{lemma}

\begin{proof}
By the above discussion and \eqref{eq:germs12}, it is enough to show that the vector space $H_{\hdot}(G)$
is finite-dimensional, for any subgroup $G\subseteq \G$. This in turn will follow from results in
\cite[Chapters III-IV]{R}. We know that the solv-lattice $\G$ is a polycyclic group, by \cite[Proposition 3.7]{R}.
Hence, $G$ must be polycyclic \cite[Remark 4.2]{R}. According to Theorem 4.28 from \cite{R}, 
$G$ contains as a normal subgroup of finite index a solv-lattice $\G_0$, with associated (compact) solvmanifold $M_0$.
A transfer argument shows then that $H_{\hdot}(G)$ is a quotient of $H_{\hdot}(\G_0)= H_{\hdot}(M_0)$, 
and we are done.
\end{proof}

We may now describe germs at $1$ of non-abelian topological Green--Lazarsfeld loci for solvmanifolds, as follows.

\begin{theorem}
\label{thm:jumpsol}
Let $M=S/\G$ be a solvmanifold and let $\iota : \GG \to \GL (V)$ be a rational representation of a complex,
semisimple linear algebraic group of rank $1$, with tangent map $\theta\colon \sl_2 \to \gl(V)$. Then 
the germ at $1$ of $\Hom (\G, \GG)$ is isomorphic to the germ at $0$ of the cone on $\PP(H^1(M))\times \PP(\sl_2)$,
and the embedded topological Green--Lazarsfeld germs are given by
\begin{equation*}
\VV^i_1(M, \iota)_{(1)} =
\begin{cases}
\emptyset & \text{if $H^i (M)=0$,}\\
\cone (\PP(H^1(M))\times V(\det \circ \theta))_{(0)}
& \text{otherwise.}
\end{cases}
\end{equation*}
\end{theorem}

\begin{proof}
Let $A$ be a Kasuya model for $M$. Theorem B from \cite{DP-ccm} allows us to replace $\Hom (\G, \GG)_{(1)}$
by $\F(A,\sl_2)_{(0)}$ and $\VV^i_1(M, \iota)_{(1)}$ by $\RR^i_1(A, \theta)_{(0)}$. The assertion about 
$\Hom (\G, \GG)_{(1)}$ follows then from Lemma \ref{lem:kflat} and the description \eqref{eq:pcoord}
of $\F^1(A,\sl_2)$, since $H^{\hdot}(A) \simeq H^{\hdot}(M)$. If $H^{i}(A) =0$,  $\RR^i_1(A, \theta)_{(0)}=\emptyset$,
by \eqref{eq:germs12}. When $H^{i}(A) \ne 0$, Lemma \ref{lem:kres} and Corollary \ref{cor:pisharp} 
together imply that $\RR^i_1(A, \theta)_{(0)}= \Pi(A,\theta)_{(0)}$. The identification of $\Pi(A,\theta)$
with the cone on $\PP(H^1(M))\times V(\det \circ \theta)$ follows as before.
\end{proof}

Note that the complete description of embedded non-abelian jump loci from Theorem \ref{thm:jumpsol}
depends only on $\theta$ and the untwisted Betti numbers of $M$, and extends the similar result on nilmanifolds 
obtained in \cite{MPPS}.

\begin{example}
\label{ex:gap}
When the Lie group $S$ is nilpotent, a remarkable result of Dixmier \cite{Dix} says that 
the untwisted Betti numbers of $M$ have no gaps, i.e., $b_i(M)\ne 0$ for all $i\le \dim M$. 
This no longer holds for solvmanifolds. Indeed, Kasuya constructed in \cite[Example 10.4]{K1},
for every $s\ge 1$, a solvmanifold $M$ of dimension $2s+2$ with the property that $b_{s+1}(M)= 0$. 
\end{example}

We close this section by examining the rank one jump loci of virtually polycyclic groups. Let $M$ be a connected
CW-complex with finite $1$-skeleton and (finitely generated) fundamental group $\pi$. Denote by
$\TT (M):= \Hom (\pi, \C^{\times})= \Hom (\pi_{\ab}, \C^{\times})$ the affine {\em character torus},
corresponding to the case $\iota=\id_{\C^{\times}}$. For $\rho\in \TT (M)$, let $\C_{\rho}$ be
the associated rank $1$ local system on $M$, identified with the right $\pi$-module $\C_{\rho}$,
where $z\cdot g= z \rho (g)$, for $z\in \C$ and $g\in \pi$. As noted in \cite{MPPS}, the rank $1$
jump loci (the characteristic varieties $\VV^i_r(M)$) are given by 
\begin{equation}
\label{eq:defchar}
\VV^i_r(M)=\{\rho \in \TT (M) \mid
\dim H_i(M, \C_{\rho})\ge r\} \, .
\end{equation}

When $M$ is aspherical, we will replace it by $\pi$ in the notation. 
Our goal is to prove the following.

\begin{theorem}
\label{thm:charpoly}
If $\pi$ is a virtually polycyclic group, then the characteristic variety $\VV^i_1(\pi)$ is finite, for all $i$.
\end{theorem}

This extends a result from \cite{MP}, where it was shown that $\bigcup_{i\ge 0} \VV^i_1(\pi)= \{ 1\}$, 
when $\pi$ is a finitely generated nilpotent group. At the same time, this gives a more precise version of 
Lemma \ref{lem:kres}, when $M=S/\G$ is a solvmanifold (aspherical, with polycyclic fundamental group $\G$)
with $\cdga$ model $A$, since  $\VV^i_r(M)_{(1)} \simeq \RR^i_r(A)_{(0)}$, by the main result of \cite{DP-ccm}.
Theorem \ref{thm:charpoly} was proved for solv-lattices by Kasuya in \cite{K2}, using a different method.

We begin by pointing out a useful virtual property, for characteristic varieties of groups.

\begin{lemma}
\label{lem:virt}
Let $f\colon K\inj G$ be the inclusion of a finite index subgroup $K$ in a finitely generated group $G$. 
Then $f^*\colon \TT (G) \to \TT (K)$ has finite fibers and sends $\VV^i_r(G)$ into $\VV^i_r(K)$,
for all $i,r$. In particular, if $\VV^i_1(K)$ is finite for all $i$, then $G$ has the same property.
\end{lemma}

\begin{proof}
Clearly, $f_{\ab}: K_{\ab} \to G_{\ab}$ has finite index image. By the exactness of the $\TT$-functor, 
$f^*$ has finite kernel, which proves the first assertion. For $\rho \in \TT (G)$, a transfer argument
(see \cite[III.9]{B}) shows that $f_\hdot \colon H_\hdot (K, \C_{f^* \rho}) \to H_\hdot (G, \C_{\rho})$ 
is surjective, whence our second claim. The last property follows from the first two.
\end{proof}

If $\pi$ is virtually polycyclic, it contains a finite index subgroup $\pi_0$ which is polycyclic with infinite 
cyclic quotients in a composition series, by \cite[Lemma 4.6]{R}. Due to Lemma \ref{lem:virt}, we may
thus assume in Theorem \ref{thm:charpoly} that $\pi$ is actually poly-$\Z$, and argue by induction on the length
of a composition series. 

The induction step goes as follows. Let $\alpha$ be an automorphism of a finitely generated group $G$. 
Denote by $G_{\alpha}$ the semidirect product $G\rtimes_{\alpha} \Z$ and consider the exact sequence
\begin{equation}
\label{eq:ext}
1\to G \stackrel{j}{\longrightarrow} G_{\alpha} \stackrel{p}{\longrightarrow} \Z \to 1\, .
\end{equation}

We are going to use, for a $G_{\alpha}$-module $U$, the Hochschild--Serre spectral sequence of 
the group extension \eqref{eq:ext} (\cite[p.~171]{B}), 
\begin{equation}
\label{eq:hserre}
E^2_{st}= H_s(\Z, H_t(G,U)) \Longrightarrow H_{s+t}(G_{\alpha},U)\, ,
\end{equation}
which collapses at the $E^2$-term. We denote by $\tau= \sigma(1)\in G_{\alpha}$ 
the lift of $1\in \Z$ to $G_{\alpha}$, via the canonical section $\sigma$ of $p$. 

We start by using trivial coefficients $\Z$, to describe character tori. We find that
$(G_{\alpha})_{\ab}\simeq C \oplus \Z$, where $C$ denotes the coinvariants 
of $\alpha_{\ab}$ acting on $G_{\ab}$. Thus, $\rho \in \TT (G_{\alpha})$
is identified with $(\chi, \lambda)\in \TT (C)\times \C^{\times}$, $\chi$ 
is identified with $j^*\rho \in \TT(G)$, and $\lambda=\rho (\tau)$. 

For $U=\C_{\rho}$, the spectral sequence \eqref{eq:hserre} collapses to the isomorphism
\begin{equation}
\label{eq:bettirho}
H_{i}(G_{\alpha}, \C_{\rho})= H_0(\Z, H_{i}(G, \C_{\chi})) \oplus  H_1(\Z, H_{i-1}(G, \C_{\chi}))\, .
\end{equation}

Moreover, the action of $1\in \Z$ on $H_{\hdot}(G, \C_{\chi})$ may be described as follows, cf.
\cite[p.~171 and p.~78-79]{B}. Note that $\alpha: G\to G$ is conjugation by $\tau$. Let
$(\alpha, \id): (G, \C_{\chi}) \to (G, \C_{\chi})$ be the automorphism in the category of coefficients 
associated to $\alpha$ and $\id_{\C}$. Then the right action of $1\in \Z$ on $H_{\hdot}(G, \C_{\chi})$
is given by $\lambda \cdot (\alpha, \id)^{-1}_{*}$. 

\begin{lemma}
\label{lem:fin}
If $\dim H_i(G,U)<\infty$ for all $i$ and every finite-dimensional $G$-module $U$, 
then the same property holds for $G_{\alpha}$.
\end{lemma}

\begin{proof}
This is a straightforward consequence of the spectral sequence \eqref{eq:hserre}. 
\end{proof}

\begin{lemma}
\label{lem:ind}
Assume that $G$ has the finiteness property from Lemma \ref{lem:fin}. 
In the above setting, $\rho\in \VV^i_1(G_{\alpha})$ if and only if either
$j^*\rho =\chi \in \VV^i_1(G)$ and  $\lambda$ is an eigenvalue of 
$(\alpha, \id)_{*i} : H_{i}(G, \C_{\chi}) \to H_{i}(G, \C_{\chi})$ or
the same conditions hold for $i-1$. 
\end{lemma}

\begin{proof}
By definition \eqref{eq:defchar}, $\rho\in \VV^i_1(G_{\alpha})$ if and only if
$H_{i}(G_{\alpha}, \C_{\rho}) \ne 0$. We infer from \eqref{eq:bettirho} that 
$H_{i}(G_{\alpha}, \C_{\rho}) \ne 0$ if and only if $1$ is an eigenvalue for 
the monodromy action of $1\in \Z$ on either $H_{i}(G, \C_{\chi})$ or
$H_{i-1}(G, \C_{\chi})$. Our claim follows.
\end{proof}

\bigskip

\noindent
{\bf Proof of Theorem \ref{thm:charpoly}.} As noticed before, we may assume that $\pi$ is poly-$\Z$.
We suppose inductively that the group $G$ from \eqref{eq:ext} has the finiteness property
from Lemma \ref{lem:fin}, and Theorem \ref{thm:charpoly} holds for $G$. We will finish the induction
by showing that both properties are inherited by $G_{\alpha}$. Lemma \ref{lem:fin} takes care of
the first property. We may use then Lemma \ref{lem:ind} to deduce that, if $\rho\in \VV^i_1(G_{\alpha})$,
then $\chi=j^*\rho$ belongs to the finite set $\VV^i_1(G)\cup \VV^{i-1}_1(G)$, and the second component of
$\rho$, $\lambda$, is among the finitely many eigenvalues of $(\alpha, \id)_{*q}$ acting on $H_{q}(G, \C_{\chi})$,
for $q=i$ or $i-1$. This completes our proof. \hfill $\Box$

\section{Lie algebras}
\label{sec:lie}

The tools developed in Section \ref{sec:gral} also enable us to describe completely the germs at $0$ of both
$\sl_2$-valued flat connections and their depth $1$ resonance subvarieties, for cochain $\cdga$'s of 
finite-dimensional Lie algebras. 

By a result of Hattori \cite{Hat}, the $\cdga$ $A^\hdot =\CC^{\hdot}(\ss\otimes \C)$ is a model for
the solvmanifold $S/\G$, when the Lie algebra $\ss$ of $S$ is completely solvable. Due to Lemmas
\ref{lem:kflat} and \ref{lem:kres}, the comparison theorem with the corresponding topological germs 
at $1$ from \cite{DP-ccm} implies that $\F(A,\sl_2)_{(0)} = \F^1 (A,\sl_2)_{(0)}$ and $A$ has trivial 
rank one resonance in degree $i$, if $H^i(A)\ne 0$. Therefore, Corollary \ref{cor:pisharp} may be used 
for the cochain $\cdga$ $A$ and $\g=\sl_2$.
In what follows, we aim at extending this approach to an arbitrary cochain $\cdga$ $A=\CC^\hdot \h$, 
assuming only that the Lie algebra $\h$ is finite-dimensional. 

\begin{remark}
\label{rem=mill}
Millionschikov proved in \cite{M} that $\bigcup_{t\ge 0} \RR^t_1 (\CC^\hdot \ss)$ is finite, for 
an arbitrary solvable Lie algebra $\ss$. By the above discussion, this is an infinitesimal analogue of Theorem 
\ref{thm:charpoly}, in the completely solvable case.
\end{remark}

We refer the reader to \cite[Ch. VII-VIII]{HS} for
basic facts related to Lie algebra (co)homology. 
Let $\h$ be a finite-dimensional Lie algebra. 
Note that $\CC^\hdot \h= (\bigwedge^{\hdot} \h^*, d)$ has the exterior algebra on the dual vector space $\h^*$
as underlying commutative graded algebra, hence $\CC^\hdot \h$ is connected and finite-dimensional. Denoting by
$H^\hdot(\h, U)$ the Lie cohomology of $\h$ with coefficients in the $\h$-module $U$, note also that $H^\hdot(\CC \h)$
is the Lie cohomology of $\h$ with trivial $\C$-coefficients, denoted simply by $H^\hdot(\h)$. Hence, 
$\F (\CC^\hdot \h)= H^1(\h)$. 

We will need a couple of results on jump loci of cochain $\cdga$'s from \cite{MPPS}. First, for any finite-dimensional 
Lie algebra $\g$, there is a natural isomorphism between the variety $\F (\CC^\hdot \h, \g)$ and the 
{\em Lie representation variety} $\rep (\h, \g)\subseteq \h^*\otimes \g= \Hom (\h, \g)$, consisting of all
Lie homomorphisms from $\h$ to $\g$. Moreover,  $\F^1 (\CC^\hdot \h, \g)$ is identified with
$\rep^1 (\h, \g):= \{ \varphi \in \rep (\h, \g) \mid \rank (\varphi) \le 1 \}$. For $\g=\C$ (the rank one case),
clearly $\rep (\h, \C)= \rep (\h_{\ab}, \C)= \Hom (\h/[\h,\h], \C)= \Hom (H_1(\h), \C)=H^1(\h)$, thus 
recovering the previous identification. For $\omega\in H^1(\h)$, we will denote by ${}_{\omega}\C$ the
corresponding rank $1$ $\h$-module. This leads to the following description of rank $1$ resonance:
\begin{equation}
\label{eq:reslie}
\RR^i_r(\CC^\hdot \h)= \{ \omega\in H^1(\h) \mid \dim H^i(\h, {}_{\omega}\C) \ge r \}\, .
\end{equation}

We may now establish the Lie analogue of Theorem \ref{thm:charpoly}, in full generality.

\begin{prop}
\label{prop:lieres}
If the Lie algebra $\h$ is finite-dimensional, then $\RR^i_1(\CC^\hdot \h)$ is finite, for all $i$.
In particular, the $\cdga$ $\CC^\hdot \h$ has trivial resonance in degree $i$, if $H^i(\h) \ne 0$.
\end{prop}

\begin{proof}
By Levi's theorem \cite[p.~250]{HS}, $\h=\ss \rtimes \g$, with $\ss$ solvable and $\g$ semisimple.
Consider the Lie extension $0\to \ss \stackrel{j}{\longrightarrow} \h \longrightarrow \g \to 0$, and 
the associated natural exact sequence  \cite[p.~238]{HS} 
\begin{equation}
\label{eq:5term}
H_2(\h)\to H_2(\g)\to  \ss/[\h, \ss] \to H_1(\h)\to H_1(\g)\to 0\, ,
\end{equation}
where $H_1(\g)$ and $H_2(\g)$ vanish by semisimplicity, cf. \cite[p.~247 and p.~249]{HS}. We infer that
the map induced on abelianizations, $j_{\ab}: H_1(\ss) \to H_1(\h)$, may be identified with the surjection  
$\ss/[\ss, \ss] \surj \ss/[\h, \ss]$; in particular, $j^*: H^1(\h) \inj H^1(\ss)$ is injective. 

For $\omega\in H^1(\h)$, the extension gives rise to the Grothendieck--Hochschild--Serre spectral sequence
\cite[p.~305]{HS}
\begin{equation}
\label{eq:grspsq}
H^s(\g, H^t(\ss, {}_{j^* \omega}\C)) \Longrightarrow H^{s+t}(\h, {}_{\omega}\C) \, .
\end{equation}
If $j^* \omega \not\in \bigcup_{t\ge 0} \RR^t_1 (\CC^\hdot \ss)$, \eqref{eq:grspsq} 
implies by \eqref{eq:reslie} that $H^{\hdot}(\h, {}_{\omega}\C)= 0$. On the other hand,
the solvability of $\ss$ forces $\bigcup_{t\ge 0} \RR^t_1 (\CC^\hdot \ss)$ to be finite, as shown
by Millionschikov in \cite{M}. Our claim follows.
\end{proof}

We know from \cite{MPPS} that $\F(\CC^\hdot \h ,\sl_2) = \F^1 (\CC^\hdot \h ,\sl_2)$, when $\h$
is nilpotent. Now, we have to prove that $\F(\CC^\hdot \h ,\sl_2)_{(0)} = \F^1 (\CC^\hdot \h ,\sl_2)_{(0)}$, in general.
We will use the standard basis of $\sl_2$, $\{ H, X_+, X_- \}$ (where $H$ is diagonal with $\eig (H)= \{ \pm 1\}$), for which
$[X_+, X_-]=H$ and $[H, X_{\epsilon}]= 2 \epsilon X_{\epsilon}$.

We start by analyzing a simple metabelian case: $\h=V\rtimes_{\alpha} \C$, where $V$ is abelian and
$\alpha \in \gl (V)$. It is easy to check that $\h$ is non-nilpotent if and only if $\alpha$ has a non-zero eigenvalue.
Choose a basis $\{ z_i\}$ of $V$ for which $\alpha$ is in Jordan normal form. For an eigenvalue $\lambda \in \eig (\alpha)$,
and an $r$-Jordan block of type $\lambda$, $V(\lambda)$, $\alpha z_1=\lambda z_1$ and 
$\alpha z_i=\lambda z_i + z_{i-1}$ for $1<i\le r$. We will use the basis $\{ z_i, u\}$ for $\h$, where
$u$ corresponds to $1\in \C$. 

For $\varphi \in \rep (\h, \sl_2)$, we know that, on $V$, $\varphi z_i= t_i Z$ with $t_i\in \C$, for some
$Z\in \sl_2$, since $V$ is abelian. We set $\varphi u=U\in \sl_2$. When $\h$ is not nilpotent, we denote by
$f$ the regular function on $\Hom (\h, \sl_2)$ defined by $f(\varphi)=\prod (\det U +\frac{\lambda^2}{4})$,
where the product is taken over the non-zero eigenvalues $\lambda \in \eig (\alpha)$. Clearly, $f(0)\ne 0$.

\begin{lemma}
\label{lem:metab}
For $\h= V \rtimes_{\alpha}\C$ as above, the following hold.
\begin{enumerate}
\item \label{m1}
Assuming that $\h$ is non-nilpotent, $\varphi \in \rep (\h, \sl_2)\setminus \rep^1 (\h, \sl_2)$
if and only if there is $0\ne \lambda \in \eig (\alpha)$ and $\epsilon \in \{ \pm 1\}$ such that
$\varphi$ has the following form (up to $\GL_2$-conjugation):
\begin{itemize}
\item
$\varphi u=\frac{\lambda}{2\epsilon} H$;
\item
$\varphi =0$, on each Jordan block $V(\lambda')$ with $\lambda'\ne \lambda$; 
\item
on each $r$-Jordan block $V(\lambda)$, $\varphi z_i= t_i X_{\epsilon}$ for $1\le i\le r$, 
with $t_i=0$ for $i<r$; 
\item
$t_r\ne 0$, for at least one such block $V(\lambda)$.
\end{itemize}
Furthermore, in this situation $f(\varphi)=0$.
\item \label{m2}
$\rep (\h, \sl_2)= \rep^1 (\h, \sl_2)$ if and only if $\h$ is nilpotent.
\end{enumerate}
\end{lemma}

\begin{proof}
Part \eqref{m1}. The rank condition on $\varphi$ translates to the fact that $\rank \{ U,Z\}=2$ plus
the property that the vector $\underline{t}= (t_1, \dots, t_r)$ is non-zero for at least one Jordan block.
Clearly, $\varphi \in \rep (\h, \sl_2)$ if and only if, for every $r$-Jordan block $V(\lambda)$, 
$\varphi \alpha z_i= \varphi [u,z_i]= t_i [U,Z]$, for $1\le i\le r$. In more detail, this means that
\begin{equation}
\label{eq:rlambda}
\lambda t_1 Z= t_1 [U,Z] \quad \text{and} \quad (\lambda t_i + t_{i-1} )Z= t_i [U,Z]\, , 
\quad \text{for} \quad 1<i\le r\, .
\end{equation}

When $\underline{t}\ne 0$, \eqref{eq:rlambda} may be solved as follows. Let $i$ be 
the first value for which $t_i\ne 0$. We deduce from \eqref{eq:rlambda} that $[U,Z] =\lambda Z$.
Since $Z\ne 0$, it follows that \eqref{eq:rlambda} is equivalent to $t_i=0$ for $1\le i<r$. 

We are left with solving the equation $[U,Z] =\lambda Z$, which implies that $\lambda\ne 0$,
since otherwise our assumption $\rank \{ U,Z\}=2$ would be violated. This in turn forces $\det U\ne 0$.
Indeed, otherwise $U$ would be nilpotent, hence $\ad$-nilpotent, which implies $\lambda=0$. Thus,
we may assume that $U=t H$ with $t\in \C^{\times}$, modulo $\GL_2$-conjugation. This implies that we may
also assume $Z=X_{\epsilon}$ for some $\epsilon \in \{ \pm 1\}$ and $t=\frac{\lambda}{2\epsilon}$.
Plainly, $[t H, X_{\epsilon}]= \lambda X_{\epsilon}$.

This completes our explicit description of $\rep (\h, \sl_2)\setminus \rep^1 (\h, \sl_2)$.
The property $f(\varphi)=0$ becomes a direct consequence of the construction of $f$, since $\det H=-1$.

Part \eqref{m2}. If $\h$ is nilpotent, the claim is proved in \cite{MPPS}.  If $\h$ is non-nilpotent, 
then $\rep (\h, \sl_2)\ne \rep^1 (\h, \sl_2)$, by Part \eqref{m1}.
\end{proof}

At the next step, we settle the solvable case.

\begin{lemma}
\label{lem:solvsl2}
Let $\ss$ be a finite-dimensional solvable Lie algebra. There is a Zariski closed subset
$W\subseteq \Hom (\ss, \sl_2)$ not containing $0$ such that $\rep (\ss, \sl_2)\subseteq \rep^1 (\ss, \sl_2) \cup W$.
In particular, $\rep (\ss, \sl_2)_{(0)} = \rep^1 (\ss, \sl_2)_{(0)}$.
\end{lemma}

\begin{proof}
We argue by induction on the length of the derived series, $\{ \ss^{(k)} \}$. If $\ss^{(1)}=\ss'=0$, 
then $\ss$ is nilpotent and we may take $W=\emptyset$. For the inductive step, we assume that 
$\ss^{(k+1)}=0$ and we consider the Lie extension $0\to V \longrightarrow \ss \stackrel{p}{\longrightarrow} \h \to 0$,
where $V=\ss^{(k)}$ is abelian and $\h= \ss/\ss^{(k)}$ satisfies $\h^{(k)}=0$. Hence,
$\rep (\h, \sl_2)\subseteq \rep^1 (\h, \sl_2) \cup W'$, with $W' \subseteq \Hom (\h, \sl_2)$
Zariski closed and $0\not\in W'$. Clearly, $p^*W' \subseteq \Hom (\ss, \sl_2)$ is
Zariski closed and $0\not\in p^*W'$. 

Let $\{ u_i \}\subseteq \ss$ be a lift of an $\h$-basis. Plainly, for each $i$, $\h_i := V \oplus \C \cdot u_i$
is a Lie subalgebra of $\ss$, and $\h_i = V \rtimes_{\alpha_i} \C$, where $\alpha_i= \ad_V (u_i)$. If
$\h_i$ is nilpotent for all $i$, we claim that  $\rep (\ss, \sl_2)\subseteq \rep^1 (\ss, \sl_2) \cup p^* \rep (\h, \sl_2)$.
Consequently, we are done by induction, since clearly $p^* \rep^1 (\h, \sl_2) \subseteq \rep^1 (\ss, \sl_2)$.
To prove the claim, suppose that $\varphi \in \rep (\ss, \sl_2)\setminus p^* \rep (\h, \sl_2)$. This implies that
$\varphi (V)$ is $1$-dimensional, since $V$ is abelian. By our nilpotence assumption, we infer that
$\varphi u_i \in \varphi (V)$ for all $i$, whence $\varphi \in \rep^1 (\ss, \sl_2)$, as claimed.

Thus, we may assume that $\h_i$ is non-nilpotent, for some $i$. For each such $i$, let $f_i$ be a regular lift to
$\Hom (\ss, \sl_2)$ of the polynomial function $f$ on $\Hom (\h_i, \sl_2)$ from Lemma \ref{lem:metab}\eqref{m1}.
Let $V(F) \subseteq \Hom (\ss, \sl_2)$ be the zero set of $F:= \prod f_i$, which clearly does not contain $0$. We
claim that we may take $W=p^*W' \cup V(F)$. 

Indeed, let $\varphi \in \rep (\ss, \sl_2)\setminus  \rep^1 (\ss, \sl_2)$ be arbitrary. If 
$\varphi \in p^* \rep (\h, \sl_2)\subseteq p^* \rep^1 (\h, \sl_2) \cup p^*W'$, we are done.
If $\varphi \not\in p^* \rep (\h, \sl_2)$, pick $i$ such that $\varphi_i \not\in \rep^1 (\h_i, \sl_2)$,
where $\varphi_i$ denotes the restriction of $\varphi$ to $\h_i$, using an argument as before. In
particular, $\h_i$ is non-nilpotent. By Lemma \ref{lem:metab}\eqref{m1}, $f_i(\varphi)=0$. 
Therefore, $F(\varphi)=0$, which completes our proof. 
\end{proof}

Now, we describe a procedure that reduces the study of germs at $0$ of arbitrary Lie representation varieties
to the case when the domain is solvable. We begin with the semisimple case.

\begin{lemma}
\label{lem:nr}
Let $\g$ be a semisimple and $\kk$ be an arbitrary finite-dimensional Lie algebra. 
Then $\rep (\g, \kk)_{(0)}= \{ 0\}$.
\end{lemma}

\begin{proof}
By Ado's theorem, $\kk$ may be embedded in $\gl (V)$, for some finite-dimensional vector space $V$.
Note that $\GL (V)$ acts by conjugation on $\rep (\g, \gl (V))$, and the $\GL (V)$-orbit of $0$ is $\{ 0\}$.
By a classical result of Nijenhuis and Richardson \cite[p.~3]{NR}, $\rep (\g, \gl (V))$ is the union of
finitely many Zariski closed $\GL (V)$-orbits, provided $\g$ is semisimple. (This is due to the fact that
$H^1(\g ,U)=0$ when $\dim U<\infty$, see \cite[p.~247]{HS}.) Our claim follows then from the obvious 
equality $\rep (\g, \kk)= \rep (\g, \gl (V)) \cap \Hom (\g, \kk)$. 
\end{proof}

For an arbitrary finite-dimensional Lie algebra $\h$, pick a Levi decomposition $\h= \ss \rtimes_{\alpha} \g$,
with $\ss$ solvable, $\g$ semisimple and $\alpha$ a Lie homomorphism from $\g$ to the Lie algebra of
derivations, $\Der (\ss)\subseteq \gl (\ss)$. Let $\kk$ be another finite-dimensional Lie algebra.
The semi-direct product structure of $\h$ implies that the map 
$$\varphi \in \Hom (\h, \kk) \mapsto (\Phi:= \varphi_{\mid \ss}, \Psi:= \varphi_{\mid \g})\in \Hom (\ss, \kk)\times \Hom (\g, \kk)$$
identifies the variety $\rep (\h, \kk)$ with the Zariski closed subset of $\rep (\ss, \kk) \times \rep (\g, \kk)$ given by the equations
\begin{equation}
\label{eq:homsemi}
\Phi (\alpha (y)x)=[\Psi y, \Phi x]\, , \quad \text{for} \quad x\in \ss \quad \text{and} \quad y\in \g \, .
\end{equation}

Denote by $\widetilde{\ss}$ the (solvable) quotient of $\ss$ by the Lie ideal generated by $\alpha (y)x$, for
$x\in \ss$ and $y\in \g$. Let $q: \ss \surj \widetilde{\ss}$ be the quotient map. Consider the regular map
\begin{equation}
\label{eq:qmap}
Q \colon \rep (\widetilde{\ss}, \kk) \hookrightarrow \rep (\h, \kk)\, ,
\end{equation}
which sends $\widetilde{\Phi}$ to $(q^* \widetilde{\Phi}, 0)$. It follows from \eqref{eq:homsemi}
that $Q$ identifies the variety $\rep (\widetilde{\ss}, \kk)$ with the Zariski closed subvariety 
$\rep (\h, \kk) \cap \rep (\ss, \kk) \times 0$. 

\begin{prop}
\label{prop:homgerms}
Let $\h$ and $\kk$ be finite-dimensional Lie algebras. The map $Q$ from \eqref{eq:qmap} 
induces an analytic isomorphism $\rep (\widetilde{\ss}, \kk)_{(0)} \simeq \rep (\h, \kk)_{(0)}$.
\end{prop}

\begin{proof}
By the above discussion, it is enough to show that, for $(\Phi, \Psi)\in \rep (\h, \kk)$ near the origin,
$\Psi =0$. This follows from Lemma \ref{lem:nr}.
\end{proof}

\begin{corollary}
\label{cor:homsl2}
For any finite-dimensional Lie algebra $\h$, $\F(\CC^\hdot \h ,\sl_2)_{(0)} = \F^1 (\CC^\hdot \h ,\sl_2)_{(0)}$.
\end{corollary}

\begin{proof}
We have to show that $\rep (\h, \kk)_{(0)}= \rep^1 (\h, \kk)_{(0)}$, when $\kk=\sl_2$. 
For $\varphi =(\Phi, \Psi)\in \rep (\h, \kk)$ near $0$, we know from Proposition \ref{prop:homgerms}
that $\Psi =0$ and $\Phi= q^* \widetilde{\Phi}$ with $\widetilde{\Phi}\in \rep (\widetilde{\ss}, \kk)$
near $0$. Lemma \ref{lem:solvsl2} implies that $\dim \im ( \widetilde{\Phi}) \le 1$, and therefore 
$\dim \im (\Phi) \le 1$. We deduce that $\varphi \in \rep^1 (\h, \sl_2)$, which completes the proof.
\end{proof}

We end with the following Lie analog of Theorem \ref{thm:jumpsol}.

\begin{theorem}
\label{thm:jumplie}
Let $\h$ be a finite-dimensional complex Lie algebra and let $\theta\colon \sl_2(\C) \to \gl(V)$ be
a finite-dimensional Lie representation with $V\ne 0$. Then $\F (\CC^\hdot \h, \sl_2)_{(0)}$ is
equal to $\{ 0\}$ when $H^1 (\h)=0$, and otherwise is isomorphic to $\cone (\PP(H^1(\h))\times \PP(\sl_2))_{(0)}$.
In the first case, $\RR^i_1( \CC^\hdot \h, \theta)_{(0)}$ is empty or equal to  $\{ 0\}$, depending on
whether $H^i (\h)$ vanishes or not. In the second case, the embedded germs of depth $1$ resonance varieties are given by
\begin{equation*}
\RR^i_1(\CC^\hdot \h, \theta)_{(0)} =
\begin{cases}
\emptyset & \text{if $H^i (\h)=0$,}\\
\cone (\PP(H^1(\h))\times V(\det \circ \theta))_{(0)}
& \text{otherwise.}
\end{cases}
\end{equation*}
\end{theorem}

\begin{proof}
By Proposition \ref{prop:lieres} and Corollary \ref{cor:homsl2}, the hypotheses of Corollary \ref{cor:pisharp}
are satisfied by $A=\CC^\hdot \h$ and $\g=\sl_2$, if $H^i (\h) \ne 0$. The proof goes as in
Theorem \ref{thm:jumpsol}, when  $H^1 (\h) \ne 0$. Finally, assume $H^1 (\h) = 0$. Then
$\F (\CC^\hdot \h, \sl_2)_{(0)}= \{ 0\}$, since $\F^1 (\CC^\hdot \h, \sl_2)= \{ 0\}$, by 
definition \eqref{eq:pmap}. The remaining assertions, on $\RR^i_1(\CC^\hdot \h, \theta)_{(0)}$,
follow then from \eqref{eq:germs12}. 
\end{proof}

\begin{ack}
We are grateful to the referee, for  pertinent remarks and suggestions.
\end{ack}

\newcommand{\arxiv}[1]
{\texttt{\href{http://arxiv.org/abs/#1}{arxiv:#1}}}
\newcommand{\doi}[1]
{\texttt{\href{http://dx.doi.org/#1}{doi:#1}}}


\begin{thebibliography}{00}

\bibitem{A} D.~Arapura,
{\em Geometry of cohomology support loci for local systems  {\rm I}},
J. Algebraic Geom. \textbf{6} (1997), no.~3, 563--597.

\bibitem{B}  K.~S.~Brown,
{\em Cohomology of groups}, Grad. Texts in Math., 
vol.~87, Springer-Verlag, New York-Berlin, 1982.

\bibitem{DHP}  A.~Dimca, R.~Hain, S.~Papadima,
{\em The abelianization of the Johnson kernel}, 
J.~Eur.~Math.~Soc. \textbf{16} (2014), 805--822.


\bibitem{DP-ann}  A.~Dimca, S.~Papadima,
{\em Arithmetic group symmetry and finiteness properties 
of Torelli groups}, Annals of Math. \textbf{177} (2013), 
no.~2, 395--423.

\bibitem{DP-ccm}  A.~Dimca, S.~Papadima,
{\em Non-abelian cohomology jump loci from an analytic 
viewpoint}, Commun. Contemp. Math. \textbf{16} (2014), 
no.~4, 1350025 (47 p). 


\bibitem{DPS-duke} A.~Dimca, S.~Papadima, A.~I.~Suciu,
{\em Topology and geometry of cohomology jump loci},
Duke Math. Journal \textbf{148} (2009), no.~3, 405--457.

\bibitem{DPS14}  A.~Dimca, S.~Papadima, A.~Suciu,
{\em Algebraic models, Alexander-type invariants, and Green-Lazarsfeld sets},
Bull. Math. Soc. Sci. Math. Roumanie \textbf{58(106)} (2015), no.~3, 257--269.

\bibitem{Dix} J.~Dixmier,
{\em Cohomologie des alg\`{e}bres de {L}ie nilpotentes},
Acta Sci. Math. Szeged \textbf{16} (1955), 246--250.

\bibitem{GL1} M.~Green, R.~Lazarsfeld,
{\em Deformation theory, generic vanishing theorems, and some conjectures of
Enriques, Catanese and Beauville}, Invent. Math. \textbf{90} (1987), 389--407.

\bibitem{GL2} M.~Green, R.~Lazarsfeld, 
{\em Higher obstructions to deforming cohomology groups of 
line bundles}, J. Amer. Math. Soc. \textbf{4} (1991), no.~1, 87--103.

\bibitem{Hat} A.~Hattori,
{\em Spectral sequence in the De Rham cohomology of fibre bundles},
J. Fac. Sci. Univ. Tokyo \textbf{8} (1960), 289--331.

\bibitem{HS} P.~J.~Hilton,  U.~Stammbach,
{\em A course in homological algebra}, Grad Texts in Math.,
vol.~4, Springer-Verlag, New York, 1971.

\bibitem{KM} M.~Kapovich, J.~Millson,
{\em On representation varieties of {A}rtin groups, projective
arrangements and the fundamental groups of smooth complex
algebraic varieties}, Inst. Hautes \'{E}tudes Sci. Publ. Math.
\textbf{88} (1998), 5--95.

\bibitem{K1} H.~Kasuya,
{\em Minimal models, formality and Hard Lefschetz properties of
solvmanifolds with local systems},
J. Diff. Geom. \textbf{93} (2013), 269--298.

\bibitem{K2} H.~Kasuya,
{\em Flat bundles and hyper-Hodge decomposition on solvmanifolds},
Int. Math. Res. Notices \textbf{19} (2015), 9638--9659.

\bibitem{K3} H.~Kasuya,
{\em Singularity of the varieties of representations of lattices in
solvable Lie groups}, J. Topol. Anal. \textbf{8} (2016), no.~2, 273--285.

\bibitem{MP} A.~Macinic, S.~Papadima,
{\em Characteristic varieties of nilpotent groups and 
applications}, in: Proceedings of the Sixth Congress 
of Romanian Mathematicians (Bucharest, 2007), 
pp.~57--64, vol.~1, Romanian Academy, Bucharest, 2009.

\bibitem{MPPS}  A.~M\u{a}cinic, S.~Papadima, R.~Popescu, 
A.~Suciu, {\em Flat connections and resonance varieties: 
from rank one to higher ranks}, 
Trans. Amer. Math. Soc. 
\textbf{369} (2017), no.~2, 1309--1343.    


\bibitem{M} D.~V.~Millionschikov,
{\em Cohomology of solvmanifolds with local coefficients and problems of the Morse-Novikov theory},
Russian Math. Surveys \textbf{57} (2002), 813--814.

\bibitem{NR} A.~Nijenhuis, R.~W.~Richardson,
{\em Cohomology and deformations in graded Lie algebras},
Bull. Amer. Math. Soc. \textbf{72} (1966), 1--29.

\bibitem{N} K.~Nomizu,
{\em On the cohomology of compact homogeneous spaces of nilpotent Lie groups},
Annals of Math. \textbf{59} (1954), 531--538.

\bibitem{PS-mrl} S.~Papadima, A.~Suciu,
{\em Jump loci in the equivariant spectral sequence}, 
Math. Res. Lett. \textbf{21} (2014), no.~4, 863--883.


\bibitem{PS14} S.~Papadima, A.~Suciu,
{\em The Milnor fibration of a hyperplane arrangement: from 
modular resonance to algebraic monodromy}, 
Proc. London Math. Soc. \textbf{114} (2017), no.~6, 961--1004.

\bibitem{PS-crelle} S.~Papadima, A.~Suciu,
{\em Vanishing resonance and representations of Lie algebras},
J. Reine Angew. Math. \textbf{706} (2015), 83--101.

\bibitem{R} M.~S.~Raghunathan,
{\em Discrete subgroups of Lie groups}, 
Ergebnisse der Math., vol.~68, 
Springer-Verlag, New York, 1972.

\bibitem{S}  D.~Sullivan,
{\em Infinitesimal computations in topology}, Inst. Hautes
\'{E}tudes Sci. Publ. Math. \textbf{47} (1977), 269--331.


\end{thebibliography}
\end{document}